\documentclass[11pt]{article}
\usepackage{amsmath,amssymb,amsthm,graphicx,subfigure,float,url}

\usepackage{tikz}

\usepackage{parskip} 
\usepackage{pdfsync}
\graphicspath{{fig/}}

\usepackage{mathrsfs} 
\usepackage[colorlinks=true]{hyperref} 
\usepackage{pdfsync}

\usepackage{enumitem}
\usepackage{dsfont} 
\usepackage[T1]{fontenc} 

\def\R{\textrm{I\kern-0.21emR}}
\def\N{\textrm{I\kern-0.21emN}}

\renewcommand{\geq}{\geqslant}
\renewcommand{\leq}{\leqslant}

\renewcommand{\geq}{\geqslant}
\renewcommand{\leq}{\leqslant}

\newcommand{\norm}[1]{\left\lVert#1\right\rVert}
 \def \equi#1{\mathrel{\mathop{\kern 0pt\sim}\limits_{#1}}}

\newcommand {\e}  {\varepsilon}

\newcommand {\Chi} {{\bf \raise 2pt \hbox{$\chi$}} }
\newcommand {\f}   {\frac}
\newcommand {\p}   {\partial}
 
\newcommand{\beq}{\begin{equation}}
\newcommand{\eeq}{\end{equation}}
\newcommand{\bea} {\begin{array}{rl}}
\newcommand{\eea} {\end{array}}
\newcommand{\bepa}{\left\{ \begin{array}{l}}
\newcommand{\eepa} {\end{array}\right.}

\newcommand{\bmu}{\begin{multline}}
\newcommand{\emu}{\end{multline}}

\newtheorem{theorem}{Theorem}  
\newtheorem{proposition}{Proposition}
\newtheorem{corollary}{Corollary}

\newtheorem{lemma}{Lemma}

\theoremstyle{definition}\newtheorem{remark}{Remark}
\title{Phase portrait control for 1D monostable and bistable reaction-diffusion equations}

\author{Camille Pouchol\thanks{\scriptsize Sorbonne Universit\'e, Universit\'e Paris-Diderot SPC, CNRS, Laboratoire Jacques-Louis Lions, Inria MAMBA Team, F-75005, Paris, France (\href{mailto:pouchol@ljll.math.upmc.fr}{pouchol@ljll.math.upmc.fr})} \; \;
Emmanuel Tr\'elat\thanks{\scriptsize Sorbonne Universit\'e, Universit\'e Paris-Diderot SPC, CNRS, Laboratoire Jacques-Louis Lions, Inria CAGE Team, F-75005, Paris, France (\href{mailto:emmanuel.trelat@upmc.fr)}{emmanuel.trelat@upmc.fr}) }\; \;
Enrique Zuazua\thanks{\scriptsize DeustoTech, Fundaci\'on Deusto, Avda Universidades, 24, 48007, Bilbao, Basque Country, Spain; 
\newline
Departamento de Matem\'aticas, Universidad Aut\'onoma de Madrid, 28049 Madrid, Spain;
\newline 
Facultad Ingenier\'ia, Universidad de Deusto, Avda. Universidades, 24, 48007 Bilbao, Basque Country, Spain; 
\newline
Sorbonne Universit\'e, Universit\'e Paris-Diderot SPC, CNRS, Laboratoire Jacques-Louis Lions, F-75005, Paris, France (\href{mailto:enrique.zuazua@deusto.es)}{enrique.zuazua@deusto.es})}}

\begin{document}

\newcounter{assum}

\maketitle

\begin{abstract}
We consider the problem of controlling parabolic semilinear equations arising in population dynamics, either in finite time or infinite time. These are the monostable and bistable equations on $(0,L)$ for a density of individuals $0 \leq y(t,x) \leq 1$, with Dirichlet controls taking their values in $[0,1]$. We prove that the system can never be steered to extinction (steady state $0$)  or invasion (steady state $1$) in finite time, but is asymptotically controllable to $1$ independently of the size $L$, and to $0$ if the length $L$ of the interval domain is less than some threshold value $L^\star$, which can be computed from transcendental integrals. In the bistable case, controlling to the other homogeneous steady state $0 <\theta< 1$ is much more intricate. We rely on a staircase control strategy to prove that $\theta$ can be reached in finite time if and only if $L<L^\star$. The phase plane analysis of those equations is instrumental in the whole process. It allows us to read obstacles to controllability, compute the threshold value for domain size as well as design the path of steady states for the control strategy.
 \end{abstract}

\section{Introduction} 
\label{Section1}
For $L>0$, $0 \leq T \leq +\infty$, we consider the following controlled reaction-diffusion equation on $(0,L) \times (0,T)$
\begin{equation}
\begin{cases}
\label{Model}
y_t  - y_{xx} = f(y), \\
y(t,0) = u(t),  \; y(t,L) = v(t),\\
y(0) = y_0.
\end{cases}
\end{equation}

where $f$ is a $C^1$ nonlinearity satisfying $f(0) = f(1) = 0$, with initial data $0 \leq y_0 \leq 1$ in $L^\infty(0,L)$. The Dirichlet controls $u$ and $v$ are measurable functions satisfying the constraints
\begin{equation}
\label{Constraints}
0\leq u(t)\leq 1,\quad 0\leq v(t)\leq 1.
\end{equation}
In such a setting, \eqref{Model} admits a unique solution in \[L^2((0,T) \times (0,L)) \cap C([0,T];H^{-1}(0,L)),\] see for instance~\cite{Lions1971}. The constraints on the controls entail  \[0 \leq y(t,x) \leq 1,\]
for all $a.e. \,(t,x) \in [0,T] \times [0,L]$, by the parabolic comparison principle~\cite{Protter1967, Ladyzenskaya1968}.
 
We will consider two types of functions.
\begin{enumerate}
\item [(H1)]
\label{Monostable}
The \textit{monostable} case: $f>0$ on $(0,1)$.  In such a case, we will also assume $f'(0)>0$.  The typical example is $f(y) = y(1-y)$.
\item [(H2)]
\label{Bistable}
The \textit{bistable} case: $f<0$ on $(0,\theta)$ and $f>0$ on $(\theta,1)$ where $0< \theta <1$. In such a case, we will also assume $f'(0)<0$ and $f'(1)<0$. The typical example is $f(y) = y(1-y)(y-\theta)$. 
\end{enumerate}
We also set 
\begin{equation}
\label{Primitive}
F(y) := \int_0^y f(z) \, dz \text{ for } y\in[0,1].
\end{equation}
In the case (H2), we will without loss of generality always assume $F(1)\geq 0$, which is equivalent to $\theta \leq \frac{1}{2}$ when $f(y)=y(1-y)(y-\theta)$. If $F(1)<0$, one can set $z=1-y$ to apply the results obtained when $F(1) \geq 0$. 

By means of appropriately chosen Dirichlet controls $u(t)$ and $v(t)$ in $L^\infty(0,T; [0,1])$ at $x=0$ and $x=L$ respectively, our goal is to control the equation towards either the steady states $0$, $1$, or in cases (H2), also towards the steady state $\theta$. 

Let us denote $a$ a generic solution of $f(y) = 0$, namely $a=0$, $a=1$ or also $a= \theta$ in the case (H2). 
Our goal is to provide controls $u$, $v$ steering the system to those homogeneous steady states. We will say that the controlled equation \eqref{Model} is
\begin{itemize} 
\item
\textit{controllable in finite time towards $a$} if for any initial condition $0 \leq y_0 \leq 1$ in $L^\infty(\Omega)$, there exist $0 \leq T < +\infty$, controls $u, \, v \in L^\infty(0,T; [0,1])$ such that $$y(T,\cdot) = a.$$ 
\item 
\textit{controllable in infinite time towards $a$} if for any initial condition $0 \leq y_0 \leq 1$ in $L^\infty(\Omega)$, there exist controls $u, \, v \in L^\infty(0,+\infty; [0,1])$ such that $$y(t,\cdot) \longrightarrow a$$ uniformly in $[0,L]$ as $t$ tends to $+\infty$. 
\end{itemize}

\paragraph{Motivations.} 
These models are ubiquitous in population dynamics (see~\cite{Aronson1978, Kanarek2010, Perthame2015}) but they also appear in other contexts, \text{e.g.} in the theory of combustion. Let us use the point of view of population dynamics to introduce the main modeling aspects. 

In case (H1), having in mind the example $f(y) = y(1-y)$, there is exponential increase of $y$ whenever $y>0$, but there is a saturation effect near $y=1$ because the full capacity of the system has been reached. In case (H2), $f$ takes negative values close to $0$ to model the fact that a minimal density $\theta$ is required for reproduction and cooperation, under which the population will die out. The state $\theta$ is unstable in the absence of diffusion, since $f'(\theta)>0$.

These models are also amenable to modeling invasion phenomena, because (when posed on the whole space)  they typically have solutions called \textit{travelling waves} in the form $y(x-ct)$ for certain speeds $c$, linking the states $0$ and $1$, see the pioneering work~\cite{Kolmogorov1937}.

For such problems, it is thus a requirement for the solution to satisfy $y \geq 0$, a condition which is fulfilled with non-negative Dirichlet boundary conditions. We might consider using controls that are above $1$, taking into account the possibility for releases at $0$ or $L$ to be above the capacity of the system. 

However, there are contexts in which $y$ is the proportion of individuals of type $A$ over the total number of individuals of types $A$ and $B$. This can be obtained as the suitable limit of a system of two reaction-diffusion equations for each type~\cite{Strugarek2016a}. Thus, we shall also impose that the controls are below $1$ to cover these cases, which will not be a restriction for the results. 

In applications, it is common to target extinction or invasion of a given population: the goal is to reach the steady state $0$ or the steady state $1$. Converging to an intermediate steady state such as $\theta$ can also be desirable if the goal is to maintain the population all over the domain, but below invasion levels. 

If one thinks of $y$ as a proportion of one species over the total number of individuals in two species, reaching $\theta$ is one way of ensuring coexistence on the whole domain $(0,L)$. Doing it is a priori a more challenging task than for $0$ and $1$ since $\theta$ is an unstable equilibrium for the dynamical system $y' = f(y)$. 
\paragraph{On the control for model \eqref{Model}.} The literature for the control of semilinear parabolic equations such as \eqref{Model} is abundant, whether it is by means of Dirichlet controls or controls acting inside the domain~\cite{Coron2007, Zuazua2007}. The typical results (for nonlinearities small enough to avoid blow-up) when such controls are unbounded is the possibility to control towards $0$ in any time $T>0$~\cite{Lebeau1995,Emanuilov1995,Fernandez-Cara2000a}, but of course at the expense of controls becoming larger and larger as $T$ becomes smaller~\cite{Fernandez-Cara2000}. 

Much effort has been recently put into studying controllability problems also in the presence of constraints on the controls, because it is a quite common assumption for applications. Such additional control constraints, and in particular non-negativity constraints, may dramatically change the types of results one can obtain~\cite{Pighin2017,Loheac2017}. Controllability to $0$ is no longer granted, and when it is, a minimal time for controllability might appear. Also note that even with unbounded controls, constraints on the state itself can lead to absence of controllability or appearances of minimal times for it to hold true~\cite{Loheac2018}. 

Finally, it is also possible to think of controlling variables that enter nonlinearly in the equation, such as the Allee parameter $\theta$~\cite{Trelat2018}. In those cases, however, the control has a much weaker effect on the equation, thereby weakening the controllability properties.

\paragraph{A simple static strategy.}  
The simplest approach to steering the system to a homogeneous steady state $a$ is by choosing constant controls $u(t) = v(t) = a$, a strategy we shall call \textit{static} in what follows. A crucial result due to Matano is that any trajectory must converge to some stationary state.
\smallskip
\begin{theorem}[\cite{Matano1978}-Theorem B]
\label{Matano}
Consider the equation
\begin{equation}
\begin{cases}
y_t  - y_{xx} = f(y), \\
y(t,0) = \bar{u}, \; y(t,L) =  \bar{v},\\
y(0) = y_0,
\end{cases}
\end{equation}
with some constant controls $0\leq \bar{u}, \bar{v} \leq 1$. Then 
$y(t, \cdot)$ converges uniformly to some \textit{stationary state} as $t\rightarrow +\infty$, \textit{i.e.}, a solution $\bar{y}$ of 
\begin{equation}
\label{Stationary}
\begin{cases}
-\bar{y}_{xx} = f(\bar{y}), \\
\bar{y}(0) = \bar{u},  \; \bar{y}(L) = \bar{v}.\\
\end{cases}
\end{equation}
\end{theorem}

This classical but nontrivial result is established thanks to the strong maximum principle, in the spirit of work that followed studying the number of oscillation points or of sign changes of solutions (lap-number, see~\cite{Matano1982}) as time evolves.

Note that the limit stationary state is not necessarily known: the above result only asserts its existence. Moreover, it is not necessarily unique. As a consequence, choosing $u(t) = v(t) = a$ will work asymptotically (independently of the initial condition) if the homogeneous solution $a$ is the only solution to the above stationary problem \eqref{Stationary} for $\bar{u} = \bar{v} = a$. 

\paragraph{Threshold for domain size and obstacles.} 
Intuitively, one can expect that if $L$ is small, $y=a$ will be the only solution to the previous stationary problem, while if $L$ is large, there might be others. It is indeed well-known that there exists a threshold for $L$ under which $a$ is the only solution, and above which there is at least one other~\cite{Lions1982}.

Let us clarify this point with $a=0$: it is proved in~\cite{Lions1982}, (Theorem $1.4.$ for the monostable case (H1), Theorem $1.5.$ for the bistable case (H2)) that there exists a positive solution $0 < z \leq 1$ to 
\begin{equation*}
\begin{cases}
- z_{xx}  =   \lambda f(z), \\
 z(0) = z(1) = 0,
\end{cases}
\end{equation*}
depending on the position of $\lambda$ with respect to some threshold. This is equivalent to $L$ being above some threshold, since after changing variables through $z(x) = y(\frac{x}{L})$, $L$ and $\lambda$ are related by $\lambda = \frac{1}{L^2}$. We denote this threshold $L^\star$ both in the monostable and bistable cases.

It is easily checked (after change of variables) that this result also applies to prove the existence of a threshold for $\theta$ in the bistable case (H2), which we will denote $L_\theta$ and which satisfies $L_\theta < L^\star$. Indeed, any non-zero solution $y$ to the stationary problem with null Dirichlet boundary conditions will reach its maximum above~$\theta$: at such a maximum point $x_0$, there must hold $f(y(x_0)) = -y_{xx}(x_0) > 0$ and thus $y(x_0)$ is in $(\theta,1)$. As a consequence, this solution will cross $\theta$ at least twice and give rise to a solution of the stationary problem with $\theta$-boundary conditions, on a smaller interval. 
%

We shall see that there also exists a threshold for $a=1$, which is infinite under the hypothesis that $F(1)\geq 0$. To unify statements, this infinite threshold will be denoted by~$L_1$. 

When $\bar{u} = \bar{v} = a$, other stationary solutions $\bar{y}$ to \eqref{Stationary} than $a$ are obstacles for the static control strategy to work, since if $y_0 \geq \bar{y}$ (resp. $y_0 \leq \bar{y}$), then $y(t, \cdot) \geq \bar{y}$ (resp. $y(t, \cdot) \leq \bar{y}$) for the solution of the controlled model with constant controls $u(t) = v(t) = a$. This a again a consequence of the parabolic comparison principle. 

Note that these obstacles also come up naturally for the construction of so-called \textit{bubbles}, \textit{i.e.}, initial conditions in the case (H2) on the whole space, which are big enough to induce invasion~\cite{Strugarek2016, Bliman2017}. 

Consequently, combining Matano's Theorem with this threshold phenomenon already yields that the static strategy with $u(t) = v(t) = a$ is such that (leaving aside the case of $L = L_a$ for the moment):
\begin{itemize}
\item 
for $L<L_a$, any initial condition converges asymptotically to $a$,
\item
for $L > L_a$, there exist some initial conditions for which the solution will not converge to $a$. 
\end{itemize}

\paragraph{Application to invasion and extinction.}  
Another application of the comparison principle shows that it suffices to consider the static strategy in the case of $a=0$ and $a=1$. We take $a=0$ to illustrate the idea. The solution $y$ of the controlled equation of \eqref{Model} is such that $y(t,x) \geq z(t,x)$ where $z$ solves the same equation but with $u(t)=v(t) =0$. Thus a given control strategy will reach $0$ in finite or infinite time if and only if the static strategy does too. 

Also, whenever $y_0 \neq 0$, there holds $z(t,x) > 0$ inside $(0,L)$ for all times, by the strong parabolic maximum principle~\cite{Protter1967}. It entails that when $y_0 \neq 0$, it is possible to reach the state $0$ only asymptotically, and the same holds for $1$. At this stage, for $a=0$ or $a=1$, we can state that the system is not controllable to $a$ in finite time, and that it is controllable in infinite time towards $a$ depending on the position of $L$ with respect to $L_a$. 

\paragraph{Designing strategies for $\theta$.}   \;
The previous reasoning shows that the steady state $\theta$ will asymptotically attract all trajectories if $L$ is small enough, more precisely if $L<L_\theta$, just by the static strategy of putting $u(t) = v(t) =\theta$ on both sides. One can then hope to reach $\theta$ in finite time, by waiting for the system to be close enough to $\theta$ in order to use a local controllability result~\cite{Russell1978, Emanuilov1995, Lebeau1995}. 

Contrarily to the case of $0$ and $1$, the static strategy might be improved for the control towards $\theta$ since controls can take values both above and below $\theta$. 
If either $0$ or $1$ attracts all trajectories, our idea is to try and use a path of steady states linking $\theta$ to $0$ (or $1$), in order to use the \textit{staircase} method inspired by~\cite{Coron2004} and its development in~\cite{Pighin2017}. It allows to steer (in finite time) any steady state to another one, as long as they are linked by a path of steady states.

\paragraph{Main results.} 
In this paper, we provide a complete understanding of controllability properties towards constant steady states for the equation \eqref{Model}, and the central tool is the phase plane analysis for the ODE $-y''=f(y)$.
First, it will provide us with a different approach to establish the existence of thresholds. We shall actually get more precise results by proving that $L_1 = +\infty$ (due to $F(1)\geq 0$ which implies that $1$ has an advantage over $0$), and that $L^\star$ is positive and can be computed explicitly as the infimum of some transcendental integrals. More precisely, we will show that
\begin{itemize}
\item
\eqref{Model} is controllable in infinite time towards $0$ if and only if $L \leq L^\star$ in the case (H1) under generic conditions on $f$, and if and only if $L<L^\star$ in the case (H2),
\item 
\eqref{Model} is controllable in infinite time towards $1$ independently of $L$ in both cases (H1) and (H2). 
\end{itemize}
Recall that, by the strong parabolic maximum principle, controllability to $0$ or $1$ is never possible within finite time.
Furthermore, $L^\star = \pi/\sqrt{f'(0)}$ under generic conditions on $f$ in the case (H1). 
In the case (H2), let us stress that our integral formula for $L^\star$ was established for cubic nonlinearities, already with phase plane analysis in~\cite{Smoller1981}, but for other purposes.

Second, phase plane analysis will also be critical in understanding the controllability properties of $\theta$. We already know from the reasonings above that $\theta$ can be reached asymptotically by the simple static strategy, which works for $L<L_\theta$. The main contribution of our paper is the design of a control strategy which works not only for $L<L_\theta$, but more generally for $L<L^\star$. More precisely, we shall prove in the case (H2) that 
\begin{equation*} 
\text{\eqref{Model} is controllable in finite time towards $\theta$ if and only if $L<L^\star$. }
\end{equation*}
The proof of this equivalence as well as the design of an appropriate control strategy are instrumentally based on the phase plane analysis of the dynamical system $-y''=f(y)$, in the region $0\leq y\leq 1$, which involves the three steady states $0$, $\theta$ and $1$. It might seem surprising that $\theta$ cannot be reached independently of the value of $L$, since controls can take values both below and above it. This is because for $L \geq  L^\star$, a non-trivial solution to the stationary problem with zero Dirichlet boundary conditions is also an obstacle for the control towards $\theta$.

Such a strategy is far from obvious due to the instability of $\theta$ for the corresponding ODE. The main idea is to use the staircase method, together with a fine analysis of the phase plane showing that there is a path of steady states linking $0$ and $\theta$ if and only if $L<L^\star$. Actually, because the controls must be non-negative, $0$ is not an appropriate steady state and we shall need to find, again by phase plane analysis, another globally asymptotically stable steady state $y_{init}$ close to $0$ such that a path of steady states still links $y_{init}$ to $\theta$. We will also explain why there is a minimal time for controllability: one cannot hope to reach $\theta$ in arbitrarily small time, and finally we will prove that among all initial conditions $0 \leq y_0 \leq 1$ in $L^\infty(0,L)$, there exists a uniform time of controllability. 

\paragraph{Outline of the paper.}
The paper is organized as follows. In Section \ref{Section2}, we focus on the case of $0$ and $1$. Phase plane analysis allows us to recover the existence of the threshold $L^\star$ and to find an formula for it, together with some estimates. The problem of controllability towards $\theta$ is investigated in detail in Section \ref{Section3}, where we first recall the staircase method before using it with the help of phase plane analysis. Finally, Section \ref{Section4} is devoted to numerical simulations confirming theoretical results and providing alternatives such as minimal time strategies. It is concluded by some byproducts and perspectives which follow from our work.  

\section{Threshold length $L^\star$ for extinction and invasion}
\label{Section2}

\subsection{A general result for invasion}

We recall that we assume $F(1)\geq 0$ (thus $0$ and $1$ do not play the same role in the bistable case).
\smallskip
\begin{proposition}
\label{ControlOne}
Whether $f$ satisfies (H1) or (H2), \eqref{Model} is controllable in infinite time towards $1$ for all $L>0$, namely $L_1 = +\infty$.
\end{proposition}

\begin{proof}
As explained in the introduction, Matano's Theorem \ref{Matano} and the parabolic comparison principle combined imply that \eqref{Model} is controllable towards $1$ in infinite time if and only if the only solution to 
\begin{equation}
\begin{cases}
- w_{xx} = f(w), \\
w(0) = 1,  \; w(L) = 1,\\
\end{cases}
\end{equation}
with $0 \leq w \leq 1$ on $[0,L]$, is the constant $1$. 
The equation $ - w_{xx} = f(w)$ is a second-order ODE which can be rewritten as $w_x = z, \, z_x = -f(w)$,
and there is a solution to the previous equation if and only if there are curves $(w(x),w'(x))$ in the phase portrait $(w,w')$ starting and ending on the axis $w=1$, which satisfy $0 \leq w \leq 1$. In both cases, (H1) and (H2), the only such a  curve is the trivial one: $w \equiv 1$. \par
For completeness, we give an analytical proof. Assume there is a such a function $0 \leq w \leq 1$ which is not identically $1$. Then there is $x_0 \in (0,L)$ such that $w$ reaches its minimum, satisfying $w(x_0) < 1$. Since $w'(x_0) = 0$, the conservation of the energy $\frac{1}{2} w'^2 + F(w)$ yields $\frac{1}{2} (w'(0))^2 + F(1) = F(w(x_0))$ which implies $F(w(x_0)) \geq F(1)$. If $f$ satisfies (H1) or (H2) with $F(1)>0$, the last inequality imposes $w(x_0) = 1$, a contradiction. If $f$ satisfies (H2) together with $F(1) = 0$, then $w'(0) = 0$. Then $w$ would solve the second-order ODE $-w_{xx} = f(w)$ with $w(0) = 1$, $w'(0) = 0$, meaning that $w$ would be identically $1$ by Cauchy-Lipschitz uniqueness, a contradiction. 
\end{proof}

\begin{remark}
\label{Lyapunov}
In the case (H1), a Lyapunov functional exists and can be used to prove convergence to $1$~\cite{Pouchol2018}. Indeed, consider the solution to
\begin{equation*}
\begin{cases}
y_t - y_{xx} = f(y), \\
y(t,0) = 1,  \; y(t,L) = 1,\\
\end{cases}
\end{equation*}
and, for $t>0$, the functional $V(t):= \int_0^L \big(y(t,x) -1 - \ln(y(t,x))\big) \, dx$. Then
$$ \frac{dV}{dt} = - \int_0^L \left(\frac{y_x(t,x)}{y(t,x)}\right)^2  dx - \int_0^L f(y(t,x)) \frac{1-y(t,x)}{y(t,x)}  \,dx \leq 0.$$
Up to our knowledge, however, no such Lyapunov functional has been exhibited in the case~(H2). 
\end{remark}

\subsection{A general result for extinction}
Let us first note that in the case (H2) and if $F(1) = 0$, the argument given for the state $1$ in the previous section works similarly for $0$, because the phase plane shows that $0$ is the only solution to 
\begin{equation}
\begin{cases}
- w_{xx} = f(w), \\
w(0) = 0,  \; w(L) = 0,\\
\end{cases}
\end{equation}
Thus, $F(1)=0$ is a particular case for which \eqref{Model} is controllable in infinite time towards $0$ regardless of $L$. We now assume $F(1)>0$ for the rest of this section. 

Let us introduce some notations. In what follows, we will need to invert the function $F$. \par
In the case (H1), $F$ is increasing, and thus its inverse $F^{-1}$ is well-defined, mapping $[0,F(1)]$ onto $[0,1]$. \par 
In the case (H2) and if $F(1)>0$, $F$ decreases from $0$ to $F(\theta)$, and then increases from $F(\theta)$ to $F(1)>0$. There is thus a unique $\theta_1 \in (\theta,1)$ such that $F(\theta_1)=0$. We choose to denote $F^{-1}$ the inverse of $F$ on $[\theta_1,1]$ which maps $[0,F(1)]$ onto $[\theta_1,1]$. If $F(1) = 0$, we set $\theta_1 = 1$.
\smallskip
\begin{proposition}
\label{FormulaThreshold}
In cases (H1) and (H2), there exists $L^\star$ such that 
\begin{itemize}
\item
if $L < L^\star$, \eqref{Model} is controllable towards $0$ in infinite time,
\item 
if $L > L^\star$, \eqref{Model} is not controllable towards $0$ in infinite time.
\end{itemize}
Furthermore, 
\begin{equation}
\label{Formula}
L^\star  =  \inf_{\alpha \in (0, F(1))} \sqrt{2} \int_0^{F^{-1}(\alpha)} \frac{dy}{\sqrt{\alpha-F(y)}}.
\end{equation}
\end{proposition}
It also holds that $L^\star>0$, which we shall prove in the next subsection as a byproduct of Proposition \eqref{PropEstimate}. What happens if $L = L^\star$, will also be clarified in the next subsection, depending on whether $f$ is of monostable type (H1) or bistable type (H2).
\begin{proof}
We know that \eqref{Model} is controllable towards $0$ in infinite time if and only if the only solution to 
\begin{equation}
\begin{cases}
- w_{xx} = f(w), \\
w(0) = 0,  \; w(L) = 0,\\
\end{cases}
\end{equation}
with $0 \leq w \leq 1$ on $[0,L]$, is the constant $0$. 
There is a non-zero solution to the previous equation if and only if there are curves $(w(x),w'(x))$ in the phase portrait $(w,w')$ starting and ending on the $w'$-axis having length $L$ exactly, with the starting and ending points different from the origin. \par
Let us parametrize such curves by their starting point $\left(0, \sqrt{2 \alpha}\right)$ where $\alpha \in (0, F(1)]$. If these curves end on the $w'$-axis, the end-point is $\left(0, -\sqrt{2 \alpha}\right)$ and we denote $L(\alpha)$ the time required for them to reach this end-point. By symmetry, this is also twice the time for this trajectory to reach the $w$-axis, at a point which we denote $y^{max}(\alpha)$. To illustrate these curves, we refer to Figure \ref{PhasePortraitMonostable} for a schematic view of the phase portrait, given in the case~(H1). 

\begin{figure}[h!]
\begin{center}
\includegraphics[scale=0.9]{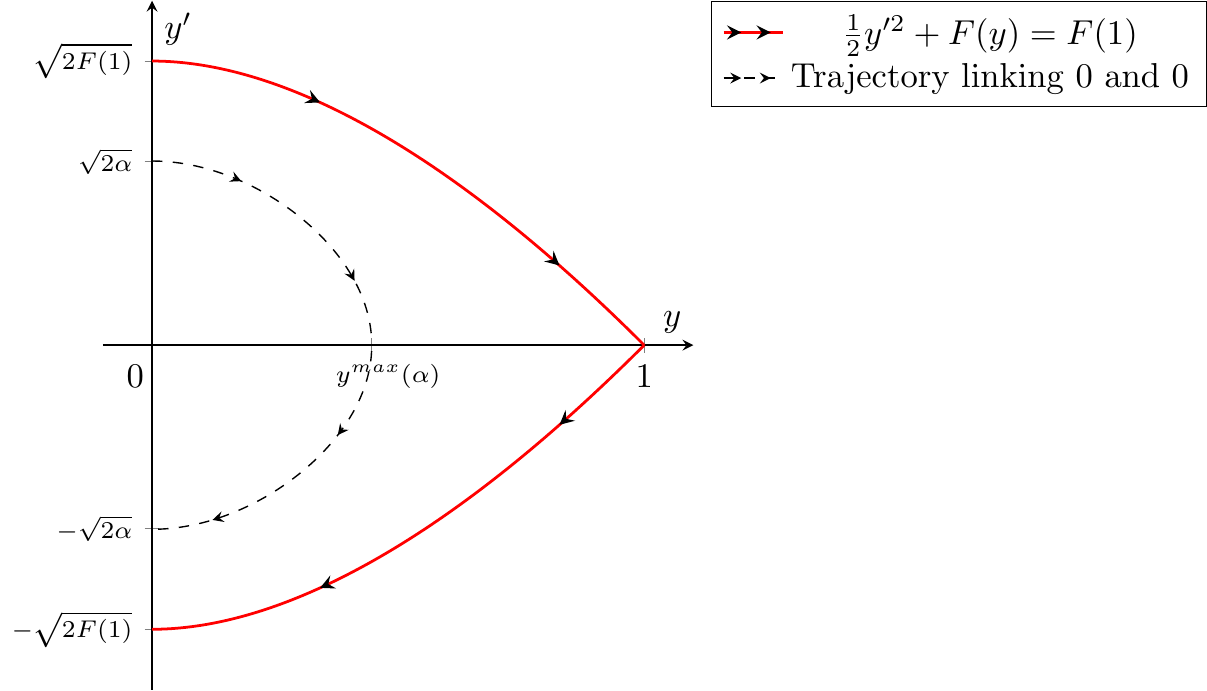}
\end{center}
\caption{Phase portrait in the monostable case (H1) (here $f(y) = y(1-y)$), with the trajectory of energy $\f12 y'^2 + F(y) =F(1)$ and an example of trajectory parametrized with~$\alpha$. }
\label{PhasePortraitMonostable}
\end{figure}

Finally, we use the fact that $y$ increases from $0$ to $y^{max}(\alpha)$, which makes of $y$ a $C^1$-diffeomorphism from $\left[0, \frac{1}{2}L(\alpha)\right]$ onto $\left[0, y^{max}(\alpha)\right]$, allowing us to compute 
$$L(\alpha) = 2 \int_0^{\frac{L(\alpha)}{2}} dz = 2 \int_0^{y^{max}(\alpha)} \frac{dy}{y'}.$$
The energy $\frac{1}{2} (y')^2 + F(y)$ is conserved along trajectories, so that $F(y^{max}(\alpha)) = \alpha$, and inverting this yields $y^{max}(\alpha) = F^{-1}(\alpha)$. We also have $y' =\sqrt{2}\sqrt{\alpha - F(y)}$, and we arrive at 
$$L(\alpha) =  \sqrt{2} \int_0^{F^{-1}(\alpha)} \frac{dy}{\sqrt{\alpha-F(y)}}.$$ 
It is easy to check that this integral is finite, except, as we will see, for $\alpha = F(1)$. Thus, $L^\star$ is well-defined. 
From this formula, one clearly infers that if $L < L^\star$, there is no curve other than $0$ linking two points on the $w'$-axis such that the corresponding trajectory satisfies $0 \leq w \leq 1$. Thus, if $L < L^\star$,  \eqref{Model} is controllable towards $0$ in infinite time. \par
To prove the second point, we compute $L(F(1)) = \sqrt{2} \int_0^1 \frac{dy}{\sqrt{F(1)-F(y)}} = +\infty$ because $F(1) - F(z) \equi{z \rightarrow 1} -\frac{f'(1)}{2} (1-z)^2$, since $F'(1) = f(1) = 0$. Consequently, $L(F(1)) = +\infty$ leading to $L(\alpha) \rightarrow + \infty$ as $\alpha$ tends to $F(1)$. Owing to the continuity of $\alpha \mapsto L(\alpha)$, this implies the existence of a non-zero stationary solution to \eqref{Stationary} (with $\bar{u}=\bar{v}=0$) and equivalently, the non-controllability of \eqref{Model} towards $0$ in infinite time, as soon as $L > L^\star$.
\end{proof}

\begin{remark}
Since $\alpha \mapsto y^{max}(\alpha)$ is increasing with $\alpha$, we can instead parametrize by $\beta:= y^{max}(\alpha)$ leading to the alternative formulae 
$$L^\star = \inf_{\beta \in (0, 1)} \sqrt{2} \int_0^{\beta} \frac{dy}{\sqrt{F(\beta)-F(y)}}$$ 
and 
$$L^\star = \inf_{\beta \in (\theta_1, 1)} \sqrt{2} \int_0^{\beta} \frac{dy}{\sqrt{F(\beta)-F(y)}},$$ 
in cases (H1) and (H2) respectively. 
\end{remark}

\subsection{Estimating $L^\star$}
\label{SubsectionEst}
Let us start by giving a global bound for $L^\star$, valid both in cases (H1) and (H2). 
\smallskip
\begin{proposition} 
\label{PropEstimate}
It holds that \begin{equation}
\label{Estimate}
L^\star \geq \frac{\pi} {\sqrt{\max_{y \in[0,1]} \frac{f(y)}{y}}}. 
\end{equation}
\end{proposition}

\begin{proof}
Let $0 \leq y_0 \leq 1$ be given in $L^\infty(0,L)$ and consider the solution to \eqref{Model} with null boundary Dirichlet values.
For $R:= \max_{y \in[0,1]}\left( \frac{f(y)}{y}\right)$, we can bound
$$y_t - y_{xx} = f(y) =\left( \frac{f(y)}{y}\right) y \leq R y\text{ on } (0,L).$$ Subsequently, $y$ is a subsolution of the equation 
\begin{equation}
\begin{cases}
z_t  - z_{xx} = Rz,  \\
z(t,0) = 0,  \; z(t,L) = 0,\\
z(0) = y_0.
\end{cases}
\end{equation}
From the comparison principle for parabolic equations, we deduce $y(t,x)\leq z(t,x)$. Now, using the Hilbertian basis of $L^2(0,L)$ formed by eigenvectors of the operator $A:z \mapsto -z_{xx} -Rz$ with Dirichlet boundary conditions, it is standard that $$\|z(t,\cdot)\|_{L^2(0,L)} \leq \|y_0\|_{L^2(0,L)} e^{-\lambda_1 t}\leq \sqrt{L} \, e^{-\lambda_1 t}$$  where $\lambda_1$ is the first eigenvalue of $A$, given by $\lambda_1 := - R + \frac{\pi^2}{L^2}$. Thus it is clear that $z(t,\cdot) \rightarrow 0$ independently of $y_0$ as soon as $L< \frac{\pi} {\sqrt{R}}$. 
\end{proof}

In the case (H1), it is possible to obtain the actual value of $L^\star$ under an additional sufficient condition on the function $f$.
\smallskip
\begin{proposition}
\label{LengthMonostable}
Let $f$ satisfy (H1) be a $C^2$ function. Further assume 
\begin{equation}
\label{HypMonostable}
f^2 \geq 2 F f' \text{ on } [0,1].
\end{equation}
Then $$L^\star = \frac{\pi}{\sqrt{f'(0)}}.$$
\end{proposition}

The hypothesis clearly applies to concave functions, whence the following corollary.
\smallskip
\begin{corollary}
Let $f$ satisfy (H1). If $f$ is concave on $[0,1]$, then $L^\star = \frac{\pi}{\sqrt{f'(0)}}.$
In particular, if $f(y) = y(1-y)$, then $L^\star = \pi$.
\end{corollary}

\begin{proof}[Proof (of Proposition \ref{LengthMonostable})]
Let us first prove that $\alpha \mapsto L(\alpha)$ is increasing on $(0,F(1))$. 
We first change variables by setting $u = \alpha F^{-1}(y)$, yielding $L( \alpha) = \sqrt{2 \alpha} \int_0^1 \frac{(F^{-1})'(\alpha u)}{\sqrt{1-u}} \, du$.
Now we compute the derivative of the previous expression for $\alpha \in (0,F(1))$ 
\begin{align*}L'( \alpha) & = \frac{1}{\sqrt{2 \alpha}} \int_0^1 \frac{(F^{-1})'(\alpha u)}{\sqrt{1-u}} \, du + \sqrt{2 \alpha} \int_0^1 \frac{u (F^{-1})^{(2)}(\alpha u)}{\sqrt{1-u}} \, dz\\
& = \frac{1}{\sqrt{2 \alpha}} \int_0^1 \frac{(F^{-1})'(\alpha u) + 2 (\alpha u) (F^{-1})^{(2)}(\alpha u)  }{\sqrt{1-u}} \, du.
\end{align*}
If $F^{-1}(z) + 2z (F^{-1})^{(2)}(z)\geq 0$ for all $z \in (0,F(1))$ our claim is proved. Computing the derivatives, we find 
$$(F^{-1})'(z) +2 z (F^{-1})^{(2)}(z) = \frac{1}{f(F^{-1}(z))}\left(1-2 z \frac{f'(F^{-1}(z))}{\left(f(F^{-1}(z))\right)^2} \right). $$
Changing variables again through $z=F(y)$, the last quantity is non-negative on $(0,F(1))$ if and only if $1 - 2 F(y) \frac{f'(y)}{f^2(z)} \geq 0$ on $(0,1)$, which is exactly the hypothesis \eqref{HypMonostable}.  \par

At this stage, we can claim that 
$$L^\star = \lim_{\alpha \rightarrow 0} L(\alpha),$$ and it remains to compute the limit. Recall that $$L(\alpha) = \sqrt{2 \alpha} \int_0^1 \frac{(F^{-1})'(\alpha u)}{\sqrt{1-u}} \, du=\sqrt{2 \alpha} \int_0^1 \frac{1}{\sqrt{1-u}} \frac{1}{f(F^{-1}(\alpha u))}\, du.$$
Since $F(y)\equi{y \rightarrow 0} \frac{F^{(2)} (0)}{2}  y^2 =\frac{f'(0)}{2}  y^2$, $F^{-1}(z) \equi{z \rightarrow 0} \sqrt{ \frac{2 z}{f'(0)} }$.
As a consequence, $f(F^{-1}(z)) \equi{z \rightarrow 0} f'(0) F^{-1}(z)  \equi{z \rightarrow 0} \sqrt{ 2 f'(0) z}$. 
Finally, we arrive at $ \frac{1}{f(F^{-1}(\alpha u))} \equi{\alpha \rightarrow  0} \frac{1}{\sqrt{2 \alpha f'(0) u }}$ leading to 
$$L(\alpha)  \equi{ \alpha \rightarrow  0}  \frac{1}{\sqrt{f'(0)}} \int_0^1 \frac{1}{\sqrt {u(1-u)}} \,du = \frac{\pi}{\sqrt{f'(0)}},$$ whence the result.
\end{proof}

From the previous result, we know we see that the infimum is attained at the boundary of $(0,F(1))$, thus we know what happens for $L = L^\star$: there is no obstacle to the convergence to $0$.
\smallskip
\begin{corollary}
Let $f$ satisfy (H1) and \eqref{HypMonostable}. Then
\eqref{Model} is controllable towards $0$ in infinite time if and only if $L \leq L^\star$.
\end{corollary}
\smallskip
\begin{remark}
Note that from the previous results, we see that exponential convergence to~$0$ is often granted, and concavity assumptions for $f$ near $0$ are critical: in such a setting, the linear part dominates and linear stability results can provide global stability.
\end{remark}

We now turn our attention towards the case (H2), for which there is no simple formula. When $F(1)=0$, we set $L^\star = +\infty$ in agreement with the result of \ref{ControlOne} which applies to $0$ when $F(1)=0$. In other words, \eqref{Model} is controllable in infinite time towards $0$ when $F(1)=0$, whatever the value of $L$. 
\smallskip 
\begin{proposition}
\label{LengthBistable}
Let $f$ satisfy (H2) and $F(1) > 0$.
Then $\alpha \mapsto L(\alpha)$ reaches a minimum at some point of~$(0,F(1)).$
\end{proposition}
\begin{proof}
We define $g(\alpha) := \frac{1}{\sqrt{2}} L( \alpha)= \int_0^{F^{-1}(\alpha)} \frac{dy}{\sqrt{\alpha-F(y)}}$ to get rid of the constant, split the integral in two on the intervals $[0,\theta_1]$ and $[\theta_1, F^{-1}(\alpha)]$, and change variables in the second integral as in the monostable case to uncover
$$
g(\alpha) = \int_0^{\theta_1} \frac{dy}{\sqrt{\alpha-F(y)}} + \sqrt{\alpha} \int_0^1 \frac{(F^{-1})'(\alpha u)}{\sqrt{1-u}}  \, du.
$$
We first note that the second integral converges to $0$ when $\alpha$ tends to $0$, while the first one converges to $\int_0^{\theta_1} \frac{dy}{\sqrt{-F(y)}} = + \infty$ because $F(y)  \equi{y \rightarrow  0} \frac{f'(0)}{2} y^2$.  
Thus, the infimum of $L(\alpha) = \sqrt{2} g(\alpha)$ is a minimum, reached inside $(0,F(1))$.
\end{proof}

%
%
%

From the previous result, we know what happens for $L = L^\star$: there is an obstacle to the convergence to $0$.
\smallskip
\begin{corollary}
Let $f$ satisfy (H2).
Then \eqref{Model} is controllable towards $0$ in infinite time if and only if $L < L^\star$.
\end{corollary}

\section{Controlling towards $\theta$ in the bistable case}
\label{Section3}
In this Section, we assume the function $f$ to be of type (H2). First recall the simple static strategy to try and reach $\theta$, which consists in setting $\theta$ on the boundary. This strategy is successful if and only if $L$ is below the threshold $L_\theta$, and as such works only for smaller domains when compared with the one we are about to introduce for $L < L^\star$ (recall that $L_\theta$ < $L^\star$).

\textbf{Estimates for $L_\theta$.} \;
Let us start by explaining how one can obtain a formula and some estimates on $L_\theta$, thanks to the results established in the previous Section \ref{Section2}. We are looking for solutions $0 \leq w \leq 1$, $w \neq \theta$ to the stationary problem  
\begin{equation}
\begin{cases}
\label{StationaryTheta}
- w_{xx} = f(w), \\
w(0) =  \theta ,  \; w(L) = \theta,\\
\end{cases}
\end{equation}
and we can look only for solutions satisfying either $w \geq \theta$ or $w \leq \theta$ when estimating the threshold $L_\theta$, since any solution taking values both above and below $\theta$, will yield solutions to \eqref{StationaryTheta} on a smaller interval. 
After change of variables $z=w-\theta$ (resp. $z = \theta -w$), we are faced with $-z_{xx} = -f(\theta-z)$ (resp. $-z_{xx} = f(\theta+z)$) with null Dirichlet boundary conditions. Thus, using the results established previously, we have the following formula for $L_\theta$: 
\[L_\theta = \inf_{\beta \in (0,1)} \sqrt{2} \left| \int_\theta^{\beta} \frac{dy}{\sqrt{F(\beta)-F(y)}} \right|. \]

It is possible to go further in estimating $L_\theta$. If we proceed as in Proposition \ref{PropEstimate}, we find \[L_\theta \geq \frac{\pi} {\sqrt{\max_{y \in[0,1]} \left| \frac{f(y-\theta)}{y-\theta}\right|}}.\]

Finally, as in Lemma \ref{LengthMonostable}, we can prove that if $f$ is of class $C^2$ satisfying $f^2(y) \geq 2 (F(y) - F(\theta)) f'(y)$ on $[0,1]$, then $L_\theta = \pi /\sqrt{f'(\theta)}$. Note that this condition requires, when $F(1) \geq 0$, that $f$ be concave near $\theta$, which is never the case for the classical cubic nonlinearity $f(y) = y(1-y)(y-\theta)$.


\paragraph{Existence of a minimal time for controllability.} 
Before proving controllability towards $\theta$ for $L< L^\star$ in finite time, a simple argument suffices to explain why it is not possible to steer the system to $\theta$ in arbitrarily small time. For simplicity, assume that $y_0 \in C([0,L])$, and first that $y_0$ is strictly above $\theta$ at least at one point in space inside $\Omega$. As usual, whatever the controls, we can write $y(t, \cdot) \geq z(t,\cdot)$, where $z(t, \cdot)$ starts from $y_0$ but with zero Dirichlet boundary conditions. 

Since the trajectory $z(t, \cdot)$ is smooth in time, it requires a positive time $t_1$ to be uniformly below $\theta$, and so if there exists a time $T$ and a control strategy such that $y(T,\cdot)  =\theta$, there must hold that $T \geq t_1$. If $y_0$ is below $\theta$ somewhere inside $(0,L)$, we argue similarly by comparing to the trajectory associated with Dirichlet boundary controls equal to $1$. 

\subsection{Control along a path of steady states}

We will say that a steady state $\bar{y}$ associated with static controls $\bar{u}$, $\bar{v}$ is \textit{admissible} if $$0 < \bar{u}, \bar{v} < 1.$$ This property will be of great importance because we shall need to make small variations around the controls $\bar{u}$, $\bar{v}$ when making use of the staircase method.

Finally, we will say that there exists a \textit{path of steady states} linking two steady states $ \bar{y}_0$ and $ \bar{y}_1$ if there is a set of steady states $\mathcal{S}$ and a continuous mapping $$\gamma : [0,1] \longmapsto \mathcal{S}$$ such that $\gamma(0) =  \bar{y}_0, \, \gamma(1) =  \bar{y}_1$, where $\mathcal{S}$ is endowed with the $C([0,L])$-topology. The corresponding 1-parameter family of controls will be denoted by $(\bar{u}_s, \bar{v}_s)_{0 \leq s \leq 1}$.


We start by giving a local exact controllability result, which holds uniformly given a family of steady states and rests on the local controllability for a single steady state, well known in the 1D case~\cite{Russell1978} and since then generalized~\cite{Emanuilov1995, Lebeau1995}, see for example~\cite{Pighin2017} for a full derivation. We stress that the controls provided by this result do not necessarily lie in $[0,1]$. We also emphasize that such a uniform result is possible because, by definition, steady states are taken to be between $0$ and $1$.
\smallskip
\begin{lemma}
\label{LocalControllability}
Let $\mathcal{S}$ be a set of steady states $\bar{y}$ associated with controls $\bar{u}, \bar{v}$. Let $T>0$ be fixed. Then there exist constants $C(T)>0$, $\delta(T)>0$ such that for all $\bar{y} \in \mathcal{S}$, for all $0 \leq y_0 \leq 1$ in $L^\infty(0,L)$ with
$$ \norm{ y_0 - \bar{y} }_\infty \leq \delta(T),$$ 
there exist controls $u,v \in L^\infty(0,T; \mathbb{R})$ such that the solution of \eqref{Model} starting at $y_0$ satisfies $$y(T, \cdot)  = \bar{y}.$$ Furthermore, $$\max\left( |u(t) - \bar{u}|, |v(t) - \bar{v}|\right) \leq C(T) \norm{ y_0 - \bar{y} }_\infty \text{ on } (0,T).$$ 
\end{lemma}

With this lemma, we can now explain the staircase method. Applied with a path of admissible steady states, it ensures that one can steer any steady state to another one by controls with values in $[0,1]$.  
\smallskip
\begin{proposition}
\label{FollowPath}
Assume that there exists a path of admissible steady states $(\bar{y}_s)_{0 \leq s \leq 1}$ associated with controls $(\bar{u}_s, \bar{v}_s)_{0 \leq s \leq 1}$. Then there exists a time $t_f>0$ and a control strategy $u,v \in L^\infty(0,t_f;[0,1])$ such that the solution of \eqref{Model} starting at $\bar{y}_0$ satisfies $$y(t_f, \cdot)  = \bar{y}_1.$$
\end{proposition}
The proof simply goes by applying a finite number of times the local controllability result along the (compact) path of steady states.
\begin{proof}
We use the previous result with time $T=1$. By continuity, $\delta_0 := \min_{0 \leq s \leq 1}(\bar{u}_s, \bar{v}_s, 1-\bar{u}_s,1- \bar{v}_s)>0$. 
We choose an integer $N$ large enough such that for all $k = 1, \ldots, N$, $$\norm{\bar{y}_{\frac{k-1}{N}}- \bar{y}_\frac{k}{N}}_\infty \leq \e,$$
where $\e$ will be defined below.   

For $k = 1, \ldots, N$, let $u_k,v_k$ be the controls in $L^\infty(0,1; \mathbb{R})$ such that the solution of \eqref{Model} starting at $\bar{y}_{\frac{k-1}{N}}$ reaches exactly $\bar{y}_{\frac{k}{N}}$ at time $1$. These controls are such that 
\begin{align*}
\max\left( \left|u_k(t) - \bar{u}_{\frac{k}{N}}\right|, \left|v_k(t) - \bar{v}_{\frac{k}{N}}\right|\right)& \leq C(1) \norm{\bar{y}_{\frac{k-1}{N}} -\bar{y}_{\frac{k}{N}}}_\infty \leq C(1) \e
\end{align*} 
on $[0,1]$. In particular, $u_k(t) \geq \bar{u}_{\frac{k}{N}} -C(1) \e >0$ for $\e$ small enough. We prove similarly that $u_k$ is bounded away from $1$, and the reasoning for $v_k$ is the same.
At this stage, it suffices to define $(u(t),v(t))= (u_k(t-k), v_k(t-k))$ for $t \in (k,k+1)$ to obtain the desired controls, with $t_f=N$.
\end{proof}

\subsection{Phase portrait in the case (H2)}
\label{PhasePortraitAnalysis}

To define a path of steady states linking an appropriate state (say $y_{init}$) to the state $\theta$, phase plane analysis is again instrumental. We will indeed consider a path of steady states $(w_s)_{0 \leq s \leq 1}$ (such that $w_0 = y_{init}$ and $w_1 = \theta$) by choosing a path of initial conditions $s \in [0,1] \mapsto (w_s(0), w'_s(0))$ in the phase plane. For some $L$ fixed, the corresponding controls are $s \in [0,1] \mapsto (w_s(0), w_s(L))$, but there is no reason in general that $0 \leq w_s(L) \leq 1$. 

However, if $0 \leq w_s(L) \leq 1$ for all $s \in [0,1]$, this method gives a path of steady states, defined by the controls $s \in [0,1] \mapsto (w_s(0), w_s(L))$. The continuity of the mapping $s \mapsto w_s$ is ensured by continuity of solutions of ODEs with respect to initial conditions. 

To ensure that the chosen path of initial conditions does not violate $0 \leq w_s(L) \leq 1$, we must analyze further elementary properties of the phase portrait in the case (H2), an example of which we depict on Figure \ref{PhasePortraitBistable}. 

\begin{figure}[H]
\begin{center}
\includegraphics[scale=1.1]{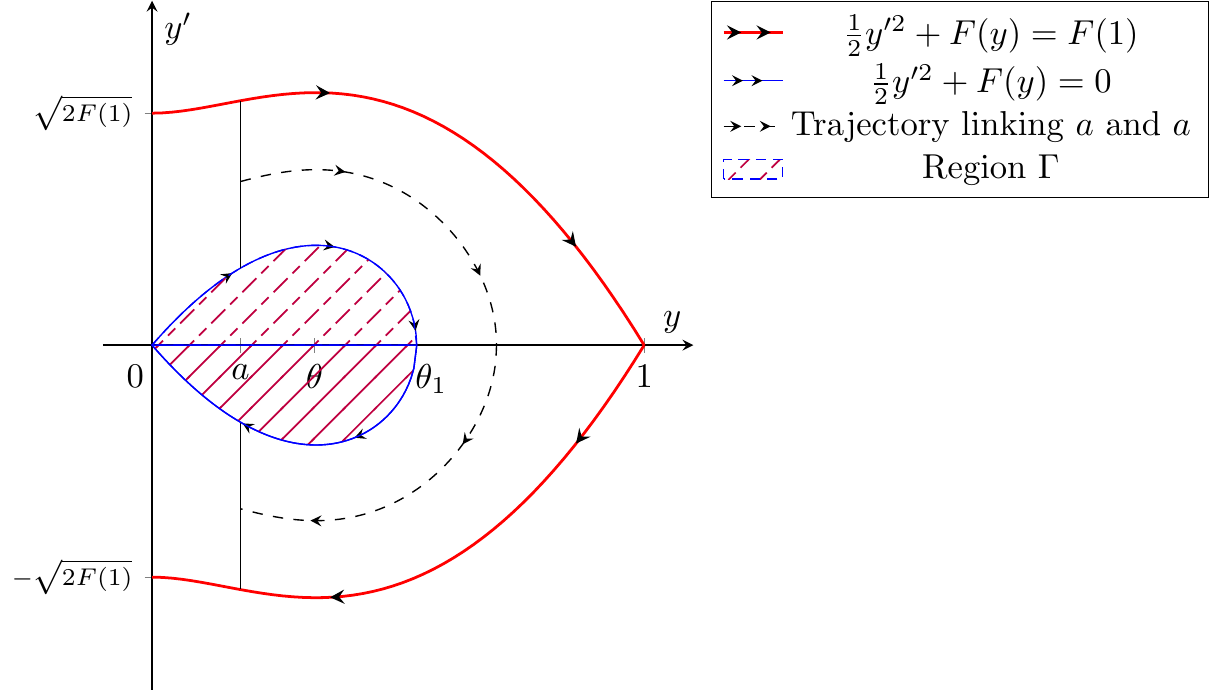}
\end{center}
\caption{Phase portrait in the bistable case (H2). (Here, $f(y) = y (1-y) (y-\theta)$, $\theta = \f{1}{3}.$) The hatched region is $\Gamma$, delimited by the trajectory of energy $\f12 y'^2 + F(y) = 0$. Also depicted: the trajectory of energy $\f12 y'^2 + F(y) =F(1)$ and an example of a trajectory starting at $a$ and ending at $a$.}
\label{PhasePortraitBistable}
\end{figure}

There are two curves of importance in the phase portrait for (H2). The first one is defined by the energy $\frac{1}{2} y'^2 + F(y) = F(1)$, while the second is a homoclinic curve and has energy $\frac{1}{2} y'^2 + F(y) = 0$. Note that if one starts with an initial condition along the first curve (resp. the second curve), it takes an infinite time (here, length), for the corresponding solution to the ODE $-w''=f(w)$ to reach $1$ (resp. $0$). The first result has been established in Proposition \ref{FormulaThreshold}, the second in Proposition \ref{LengthBistable}. With the notations of these propositions, this is because $L(\alpha)$ tends to $+\infty$ when $\alpha$ tends to $0$ or $F(1)$. 

We define $\Gamma$ to be the region defined by the set of points $(x,y)$ such that $|y| \leq \sqrt{-2F(x)}$, that is, those delimited by the homoclinic curve. The important result in what follows is that any initial condition $(w(0), w'(0))$ inside $\Gamma$ is such that the corresponding trajectory $w(x)$ remains indefinitely between $0$ and $1$ (actually, between $0$ and $\theta_1$). 

Finally, let us fix some $a \in [0,1]$. We look at all the trajectories starting with $w(0) = a$ and outside the interior of $\Gamma$, namely with $\sqrt{-2F(a)} \leq w'(0) \leq \sqrt{2(F(1)-F(a))}$. We define $L^a$ to be the minimal time for such trajectories to reach $a$ again. Note that with this definition, we clearly have $L^0 = L^\star$.

\subsection{The control strategy induced by phase plane analysis}

Let us now define the control strategy, which works not only for $L< L_\theta$ but more generally for $L < L^\star$, based on the staircase method. The core idea is to find a path of steady states between $0$ and $\theta$, which, as we shall see, is possible if and only if $L<L^\star$. However $0$ is not admissible so that we must instead resort to another close admissible steady state. We will build an admissible steady state $y_{init}$ such that
\begin{itemize}
\item $y_{init}$ can be reached asymptotically for any initial condition, 
\item there exists a path of admissible steady states linking $y_{init}$ to $\theta$.
\end{itemize}
The key lemma in order to obtain such a state is the following. 
\smallskip
\begin{lemma}
\label{Steady}
Let $L< L^\star$. 
Then for any $\e< \theta_1$ small enough, the solutions $0 \leq w \leq 1$ of
\begin{equation}
\begin{cases}
\label{StationaryEpsilon}
- w_{xx} = f(w), \\
w(0) =  \e ,  \; w(L) = \e.\\
\end{cases}
\end{equation}
are in $\Gamma$, namely they must be such that $|w'(0)| \leq \sqrt{-2F(\e)}$.
\end{lemma}
At this stage, we do not know that $y_{init}$ is unique, a fact which is not necessary for the proof of the next theorem, but we shall clarify this point in the next subsection. 

\begin{proof}
With the notations of Subsection \ref{PhasePortraitAnalysis}, $L^\varepsilon$ tends to $L^0 = L^\star$ when $\e$ tends to $0$, and thus we can choose $\e$ small enough such that $L < L^\e < L^\star$.
Consequently, by the very definition of $L^\e$, there is no solution to \eqref{StationaryEpsilon} other than those in $\Gamma$. 
\end{proof}
\smallskip
\begin{theorem}
\label{ControlTheta}
\eqref{Model} is controllable towards $\theta$ in finite time (or infinite time) if and only if~$L < L^\star$. 
\end{theorem}
\begin{proof}
We fix $L < L^\star$ and some initial data $0 \leq y_0 \leq 1$ in $L^\infty(0,L)$. 
Assume that $\e>0$ is small enough so that the conclusions of Lemma \ref{Steady} hold true. The idea is to first use static Dirichlet controls $u(t) = v(t) = \e$ for a long time, because Lemma \ref{Steady} ensures that the trajectory will converge to some steady state $y_{init}$ in $\Gamma$, independently of the initial condition. Such a steady state can then be reached exactly because of the local controllability result. Finally, the fact that $y_{init}$ is in $\Gamma$ allows us to find a path of steady states linking it to $\theta$, so that it remains to use the staircase method. 

\textit{First step.} We start by approaching a steady state $y_{init}$ in $\Gamma$.
Consider the equation
\begin{equation*}
\begin{cases}
y_t  - y_{xx} = f(y), \\
y(t,0) = \e,  \; y(t,L) = \e,\\
y(0) = y_0.
\end{cases}
\end{equation*} 
Then, by Theorem \ref{Matano}, the solution must converge to a steady state with Dirichlet boundary conditions $(\e, \e)$. By Lemma \ref{Steady}, this is some state in $\Gamma$, which we denote $y_{init}$. 
In particular, for any $\eta > 0$, there exists $t_0 > 0$ such that for $t \geq t_0$, $\norm{ y(t,\cdot) - y_{init} }_\infty \leq \eta$. 
Thus, we start by taking $u(t) = \e$, $v(t) =\e$ on $(0,t_0)$ ($\eta$ and the corresponding $t_0$ will be fixed appropriately in the next step).

\textit{Second step}.
We now make use of Lemma \ref{LocalControllability} with for example time $1$ and choosing $\eta$ (and corresponding $t_0$) such that $C(1) \eta$ is small enough for $\e - C(1) \eta  > 0$ to hold. This provides controls $\tilde{u}, \tilde{v}$ in $L^\infty(0,1;[0,1])$ such that defining $u(t) = \tilde{u}(t-t_0)$, $v(t) = \tilde{v}(t-t_0)$ on $(t_0, t_0+1)$, we have $y(t_0 +1, \cdot) = y_{init}$.

\textit{Third step}. We build a path $c$ of initial conditions linking the initial conditions associated with $y_{init}$, \textit{i.e.} $(\e,y'_{init}(0))$, and $\theta$, \textit{i.e.}, $(\theta,0)$. The simplest choice is the straight line, illustrated by Figure \ref{PathSteadyStates} below. 

\begin{figure}[h!]
\begin{center}
\includegraphics[scale=0.9]{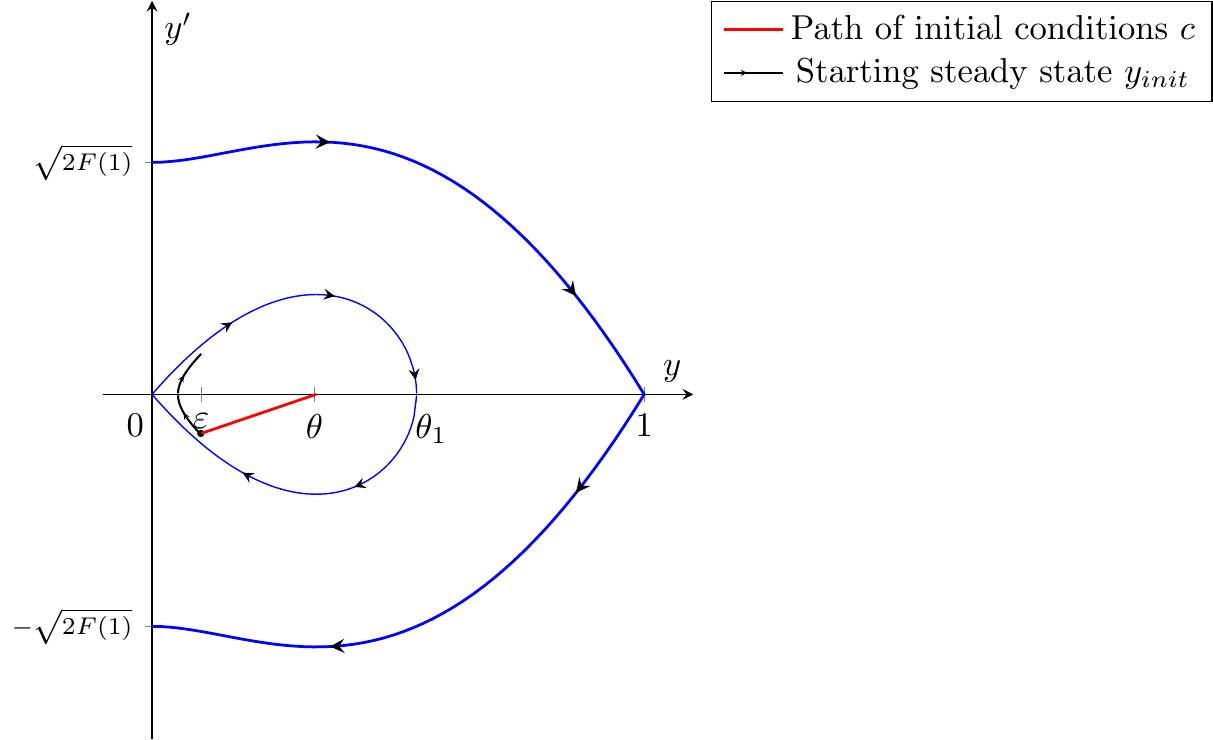}
\end{center}
\caption{The path $c$ of initial conditions, from those of $y_{init}$ (black dot) and those for $\theta$. The state $y_{init}$ is also depicted. }
\label{PathSteadyStates}
\end{figure}

We denote $\gamma$ the path of admissible steady states associated with $c$, and it now just remains to follow this path: by Theorem \ref{FollowPath}, there exist a time $T_0$ and controls $u_0, v_0$ in $L^\infty(0,T_0;[0,1])$ bringing $y_{init}$ to $\theta$. We set $u(t) = u_0(t-(t_0 +1))$, $v(t) = v_0(t-(t_0 +1))$ on $(t_0 +1, t_0 +1 +T_0)$ and $T= t_0 + 1 +T_0$. The controls $u$ and $v$ are indeed such that $y(t, \cdot)$ reaches exactly $\theta$ at time $T$.

We now prove the converse and assume $L\geq L^\star$. We already saw in the Introduction that if there exists a nontrivial solution $0 \leq w \leq 1$ to 
\begin{equation*}
\begin{cases}
- w_{xx} = f(w), \\
w(0) = 0,  \; w(L) = 0,\\
\end{cases}
\end{equation*}
it satisfies $w > \theta$ somewhere inside $(0,L)$. As already pointed out when it came to controlling towards $0$, for any control strategy $u(t)$, $v(t)$, the solution of \eqref{Model} with $y_0 \geq w$ satisfies 
$y(t, \cdot) \geq w$. If we had found a control strategy bringing us in finite (or infinite time) towards $\theta$, we would have $w \leq \theta$, a contradiction.
\end{proof}

\begin{remark}

The path $c$ is a path of initial conditions $(w_s(0), w_s'(0))$, indexed by $0\leq s \leq 1$. To clarify the associated steady states, we depict a typical example in Figure \ref{BoundaryControlsPath}. In the previous proof, the control on the left $u_s=w_s(0)$ has been chosen to increase from~$\e$ to~$\theta$. We find that the corresponding control on the right $v_s = w_s(L)$ rapidly takes values above~$\theta$. 
 
 \begin{figure}[h!]
\begin{center}
\includegraphics[scale=0.65]{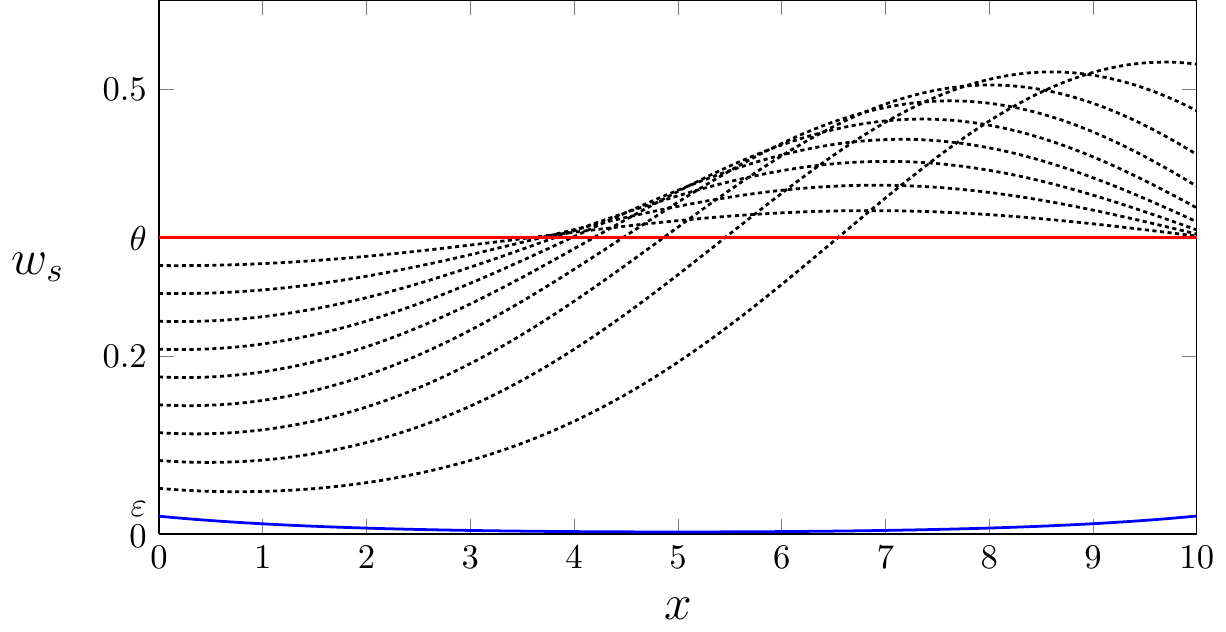}
\end{center}
\caption{Path of steady states $(w_s)$ from $y_{init}$ (in blue) to $\theta$ (in red), for $L=10$, $f(y) =y(1-y)(y-\theta)$, $\theta = \frac{1}{3}$, $\e = 0.02$.}
\label{BoundaryControlsPath}
\end{figure}

Assume that $L_\theta \leq L < L^\star$ and that the non-homogeneous solution to the stationary problem with $\theta$-boundary conditions, say $z$, satisfies $z \leq \theta$. Then a path of steady states with controls $u_s$ and $v_s$ both below $\theta$ would not work because $z$ would be an obstacle for a trajectory starting from $y_0 \leq z$. This explains why some controls $v_s$ on the right are above~$\theta$.
\end{remark}

\subsection{Uniform time of controllability}

Let us now use this control strategy to prove that there are no initial conditions requiring an arbitrarily long (finite) time to be brought to $\theta$. In other words, denoting $T_{min}(y_0)$ the minimal time for some initial condition $0\leq y_0\leq 1$ in $L^\infty(0,L)$ to be controlled to $\theta$, we have the following proposition.
\smallskip
\begin{proposition}
\label{UniformTime}
For $L<L^\star$, there exists a uniform time $T>0$ below which all initial conditions $0\leq y_0\leq 1$ in $L^\infty(0,L)$ are controllable to $\theta$, namely
\begin{equation}
\sup\, T_{min}(y_0) < \infty, 
\end{equation}
where the supremum is taken over all $0\leq y_0\leq 1$ in $L^\infty(0,L)$. 
\end{proposition}

\begin{proof}
Let $L < L^\star$ be fixed. We will make use of the control strategy of Theorem \ref{ControlTheta}. Let us remark that the time required by the second step (exact controllability to $y_{init}$) and by the third step (following the path of steady states) are independent of $y_0$. To conclude, we must analyse whether the time for the first step (approaching $y_{init}$) can be taken to be uniform over all possible initial conditions. This relies on the following lemma, which clarifies the uniqueness of $y_{init}$. 
\smallskip
\begin{lemma}
\label{UniqueInit}
There is only one solution $0 \leq y_{init} \leq 1$ to the stationary problem 
\begin{equation*}
\begin{cases}
- w_{xx} = f(w), \\
w(0) = \e,  \; w(L) = \e,\\
\end{cases}
\end{equation*}
for $\e$ small enough. 
\end{lemma}
Let us temporarily assume this result and explain how it concludes the proof. Recall that the second step requires to be close enough to $y_{init}$ so that the local exact controllability result can be used with controls lying in $[0,1]$. As in the previous proof, we characterize the corresponding neighborhood in $L^\infty(0,L)$ by $\eta$, that is, the first step is stopped when $\norm{y(t,\cdot) - y_{init}}_{L^\infty} \leq \eta$.  
We take the two extremal initial conditions $y_0^0 = 0$ and $y_0^1 = 1$, and denote the corresponding trajectories $y^0$ and $y^1$, with controls $u(t) = v(t) =\e$. We know that these trajectories both converge to the unique solution $y_{init}$ of \eqref{UniqueInit}, so that we can choose $t^\star$ such that for $t\geq t^{\star}$, both  $\norm{y^0(t,\cdot) - y_{init}}_{L^\infty} \leq \eta$ and $\norm{y^1(t,\cdot) - y_{init}}_{L^\infty} \leq \eta$ hold. 

For any initial condition $y_0^0 = 0 \leq y_0 \leq 1 =y_0^1 $, the parabolic comparison principle yields that the corresponding trajectory $y$ with controls $u(t) = v(t) =\e$ satisfies $y^0 \leq y \leq y^1$ for all times. In particular, we have $\norm{y(t,\cdot) - y_{init}}_{L^\infty} \leq \eta$ for $t\geq t^\star$. This proves that the first step takes a uniform time (bounded by $t^\star$) to bring any initial condition to the prescribed neighborhood of $y_{init}$. 
\end{proof}
We now complete the proof by proving Lemma \ref{UniqueInit}. 
Let us first prove that for $\e$ small enough, any solution of the previous problem will be such that $w \leq \e$. To do so, we shall prove that a solution with $w \geq \e$ exists only for values of $L$ which tend to $+\infty$ e as $\e$ tends to~$0$. Indeed, we already know by Lemma~\ref{Steady} that these solutions are curves in the phase plane lying in the region $\Gamma$ for $\e$ small enough. If $w \geq \e$, such a solution must be associated to some initial condition $w(0) = \e$, $0 \leq w'(0) \leq -\sqrt{-2F(\e)}$ (not $w'(0) < 0$, since otherwise $w$ is below $\e$ close to $x=0$). As $\e$ tends to $0$, these curves tend to the homoclinic curve delimiting the region $\Gamma$, which links $0$ to $0$ but in infinite time. Thus, for $\e$ small enough there are no solutions $w \geq \e$ to the stationary problem for some fixed $L$. 

At this stage, it remains to prove that there exists only one solution $0 \leq w \leq \e$ of the stationary problem for $\e$ small enough. Assume that there are two, $w_1$ and $w_2$, so that the difference $z=w_1 -w_2$ satisfies the elliptic equation $-z_{xx} = f(w_1) - f(w_2)$, with $z(0) = z(L) =0$. We choose $\e$ small enough so that $f$ is decreasing on $[0,\e]$. First assume $z'(0)>0$. 
Then $z'>0$ at least locally around $0$. If $z'> 0$ on the whole $(0,L)$, then $\e = w_1 (L) > w_2(L)=\e$  which is not possible. Thus there must be some $x_0 \in (0,L)$ such that $z'(x_0) = 0$, and we may choose it to be the first zero of $z'$. At $x_0$ it must hold that $z_{xx}(x_0) \leq 0$. However, on $(0,x_0]$, we have $z(x) = w_1(x) - w_2(x) > 0$. 
Because both $w_1$ and $w_2$ are below $\e$ and due to the monotonicity  of $f$, we obtain at $x_0$ $$- z_{xx}(x_0) = f\left(w_1(x_0) \right) - f(w_2(x_0)) < 0,$$ a contradiction. 

If $z'(0) < 0$, the reasoning is the same as before to get a contradiction. Hence, we necessarily have $z'(0) = 0$: $w_1$ and $w_2$ have the same derivative at $0$. By Cauchy-Lipschitz uniqueness, $w_1 = w_2$. 
\qed

\section{Numerical simulations, comments and perspectives}
\label{Section4}
\subsection{A numerical optimal control approach}
\label{Section4.1}
We consider the case (H2), and look for numerical control strategies to reach the state $\theta$ with the goal of both

\begin{itemize} 
\item illustrating the theoretical results,
\item investigating alternative strategies to the staircase one obtained by phase plane analysis.
\end{itemize}
To this end, we consider the following optimal control problem for some final time $T>0$: $$\text{minimize} \;\; C_T(u,v) =  \norm{y(T,\cdot) -\theta}_{L^2(0,L)}^2$$
over controls $u,v \in L^\infty(0,T; [0,1])$, and where $y$ solves \eqref{Model}.  

We are interested in seeing whether, for a given $L>0$, we can find some $T>0$ such that this optimal control problem leads to a very small cost: this will correspond to a strategy such that $y(T, \cdot)$ is very close to $\theta$. We do not need to reach $\theta$ exactly because we know that, once very close to it, there is a control strategy to reach it exactly, given by Lemma~\ref{LocalControllability}. In some instances, we will also force the controls to be equal to $\theta$ to illustrate when this control strategy suffices to reach $\theta$. 

To study this optimal control problem from a numerical point of view, we use direct methods. In a few words, the idea is to discretize the whole problem both in time and space, through discretization parameters $N_t$ and $N_x$, and to solve the resulting high but finite-dimensional optimization problem. This last step is done through the combination of automatic differentiation softwares (with the modeling language AMPL, see~\cite{Fourer2002}) and expert optimization routines (with the open-source package IpOpt, see~\cite{Waechter2006}). 

All the numerical experiments will be led with 
\begin{equation*}
f(y) = y(1-y)(y-\theta),~~~ \theta = \f 13.
\end{equation*}
In this section, we take 
\[y_0 = 0.1 \f x L + 0.8 \left(1- \f x L\right), \]
and 
\begin{equation*}
N_x = 60, ~~~ N_t = 400.
\end{equation*}
With this choice of function $f$, $\theta$, using the formula for $L^\star$, we find numerically $L^\star \approx 10.43$. As for the threshold $L_\theta$, we find $L_\theta \approx 6.29$.   

We start by taking $L=5 < L_\theta$ and impose $\theta$ on the boundary. For $T=20$, we indeed find that this is enough to approach $\theta$, see Figure \ref{cas1}. 

\begin{figure}[h!]
\begin{center}
\includegraphics[scale=0.58]{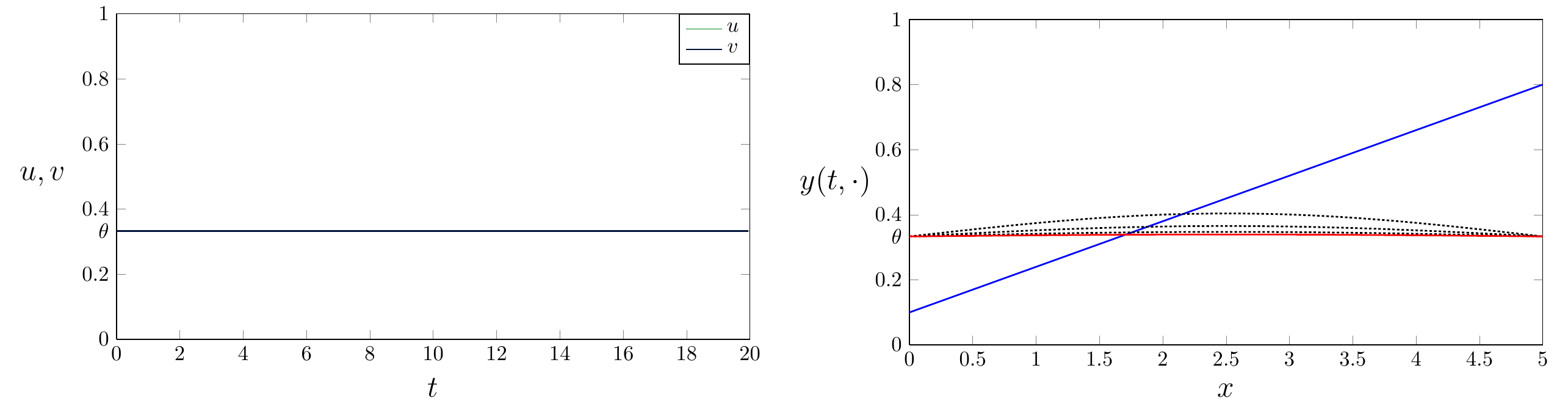}
\end{center}
\caption{Static strategy $u=v = \theta$ and resulting state $y(t, \cdot)$ from time $0$ (in blue), to $T$ (in red) with intermediate times $\frac{T}{4}$, $\frac{T}{2}$, $\frac{3T}{4}$ (in dashed line). Here, $L = 5< L_\theta$ and $T=20$.}
\label{cas1}
\end{figure}

For $L_\theta < L=8< L^\star $ and $T=20$ (or larger final times), the static strategy is not enough as already known theoretically and evidenced by the upper graphs of Figure \ref{cas2}. The lower graphs show the optimal control, as obtained numerically, to reach $\theta$: the interesting feature is that it oscillates very quickly around $\theta$ near the final time $T$. This is a common feature when controlling a heat equation to zero~\cite{Lions1988}. Also worth mentioning is the fact that controls take small values for a long time, which is reminiscent of the first long phase of our staircase strategy with $u(t) = v(t) =\e$ for a small $\e$.

\begin{figure}[h!]
\begin{center}
\includegraphics[scale=0.58]{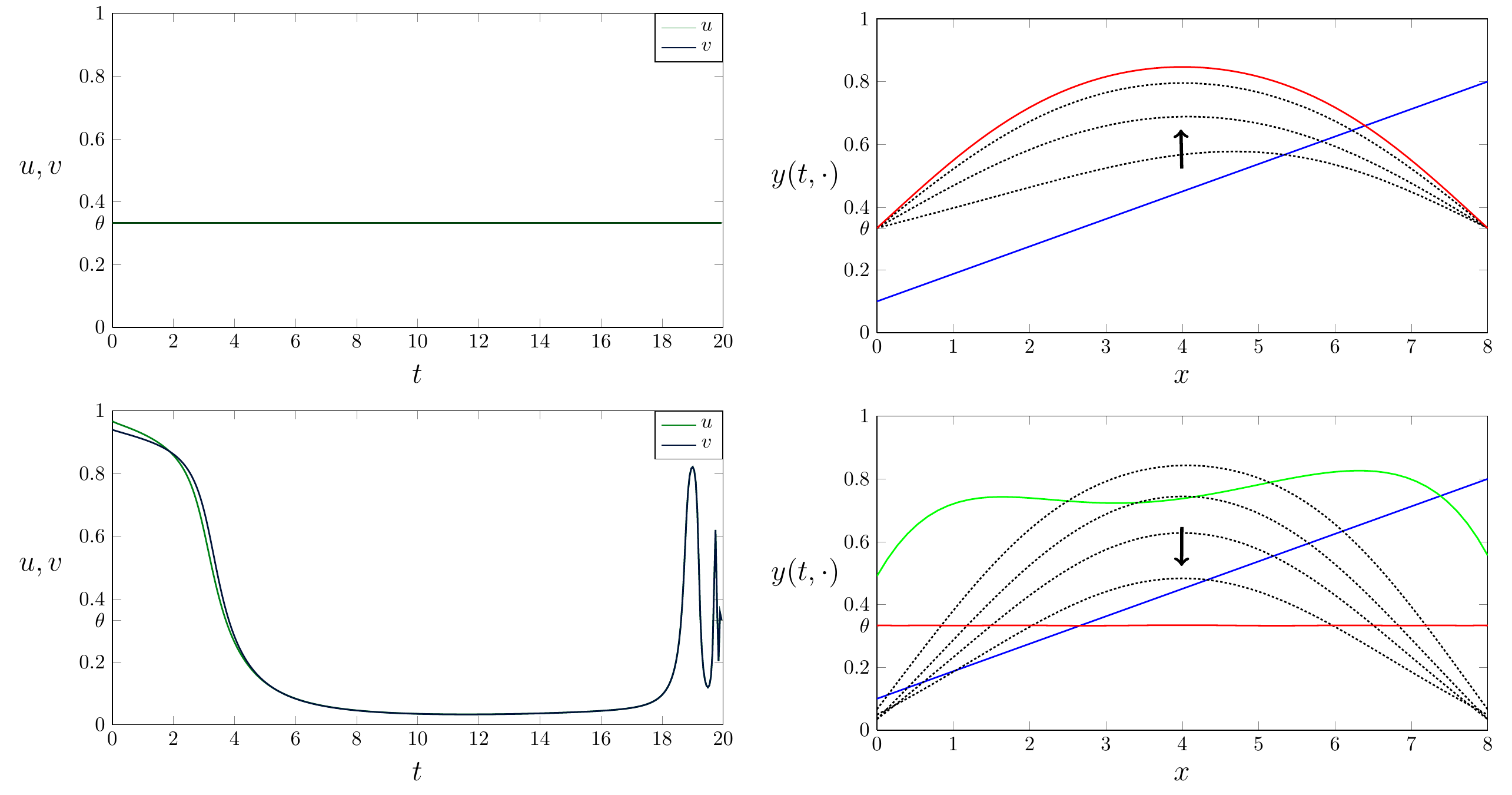}
\end{center}
\caption{Static strategy $u=v = \theta$ (upper) or optimal controls $u$, $v$ (lower) and resulting state $y(t, \cdot)$ from time $0$ (in blue), to $T$ (in red) with intermediate times $\frac{T}{4}$, $\frac{T}{2}$, $\frac{3T}{4}$ (upper, in dashed line) or $\frac{T}{6}$ (in green), $\frac{2T}{6}$, $\frac{3T}{6}$, $\frac{4T}{6}$, $\frac{5T}{6}$ (lower, in dashed line). Here, $L_\theta < L=8< L^\star $ and $T=20$.}
\label{cas2}
\end{figure}

For $L = 12 > L^\star$ and even for a large final time $T=100$, the control strategy minimizing the cost does not bring the final state close to $\theta$, see Figure \ref{cas3}. One can see that the control is close to $0$ for a long time, trying to bring the solution down but it remains blocked by a non-zero solution to the stationary problem with zero Dirichlet boundary conditions.
\begin{figure}[h!]
\begin{center}
\includegraphics[scale=0.58]{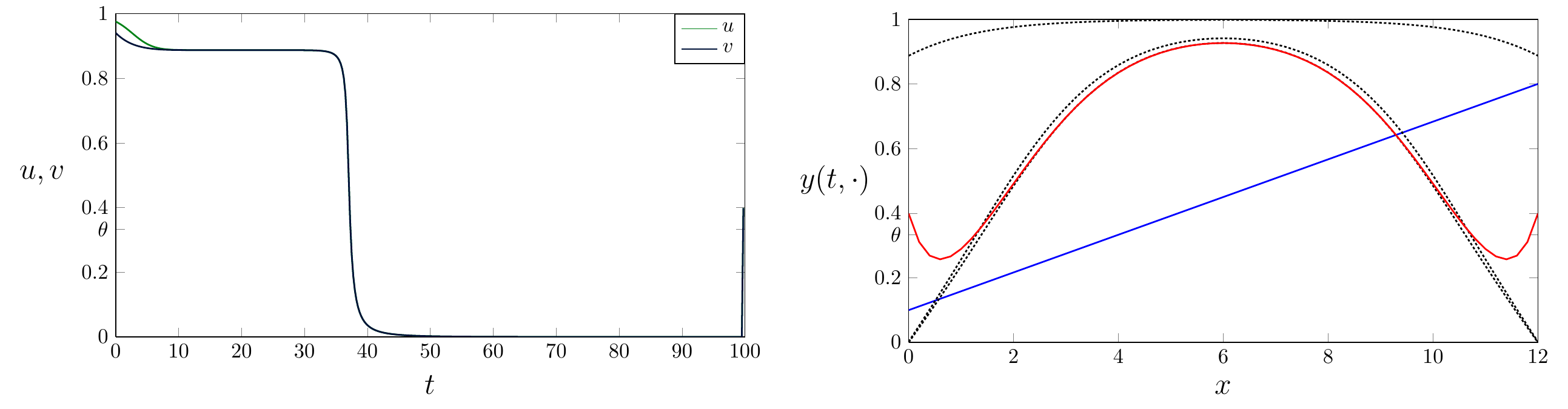}
\end{center}
\caption{Optimal controls $u$, $v$ and resulting state $y(t, \cdot)$ from time $0$ (in blue), to $T$ (in red) with intermediate times $\frac{T}{4}$, $\frac{2T}{4}$, $\frac{3T}{4}$ (in dashed line). Here, $L = 12 > L^\star$ and $T=100$.}
\label{cas3}
\end{figure}

\paragraph{About taking the same Dirichlet controls $u=v$.} 
One important feature reflected by these simulations is that the optimal controls $u$ and $v$ are actually very close to one another, almost equal after some time. Further simulations (see also the next section) performed with $u=v$ indeed indicate that it is possible to design a control strategy with $u=v$ to reach $\theta$, whenever $L< L^\star$. It remains an open problem to prove it, because we stress again that the strategy developed in Section \ref{Section3} is such that $u \neq v$.

\subsection{Control in minimal time}
\label{Section4.2}

We numerically investigate the minimal time problem, which is well-posed as proved in Section \ref{Section3}, \textit{i.e.}, we consider the optimal control problem $$\text{minimize} \;\; t_f$$
over controls $u,v \in L^\infty(0,T; [0,1])$, and where $y$ solves \eqref{Model} together with $$y(t_f, \cdot) = \theta.$$ 

As before, we discretize the whole problem to estimate the minimal time and corresponding optimal strategy. All simulations of this section are conducted with  $L = 8 < L^\star$, and the same initial condition $y_0 = 0.1 \f x L + 0.8 \left(1- \f x L\right)$ as in the previous section. The corresponding results are reported in Figure \ref{MinimalTime_TwoSides}. 

For the initial condition $y_0$, we approximately find $t_f \approx 5.2$. The optimal controls are bang-bang, \textit{i.e.,}, they take only the extremal values $0$ and $1$, except near $t_f$. They are identically equal to $0$ up until $t\approx 3.2$ and then oscillate more and more rapidly. Thus, we conjecture that the optimal controls have an infinite number of switchings near the final time $t_f$, a phenomenon called \textit{chattering}. Note that the discretization parameters are here $N_x = 100$ and $N_t = 1000$, which is necessary for a good approximation of the behavior near the final time.

At the end of the phase with both controls at $0$, $y(t,\cdot)$ is not close to $0$, as evidenced by the state in green in Figure \ref{MinimalTime_TwoSides}. Thus, the second phase does not start on a stationary state since we know that the only stationary state with null boundary conditions is $0$ for $L<L^\star$. It is also does not seem that $y(t, \cdot)$ is close to some path of steady states during the chattering phase. Consequently, simulations indicate that the staircase strategy is not the minimal time strategy.

\begin{figure}[h!]
\begin{center}
\includegraphics[scale=0.58]{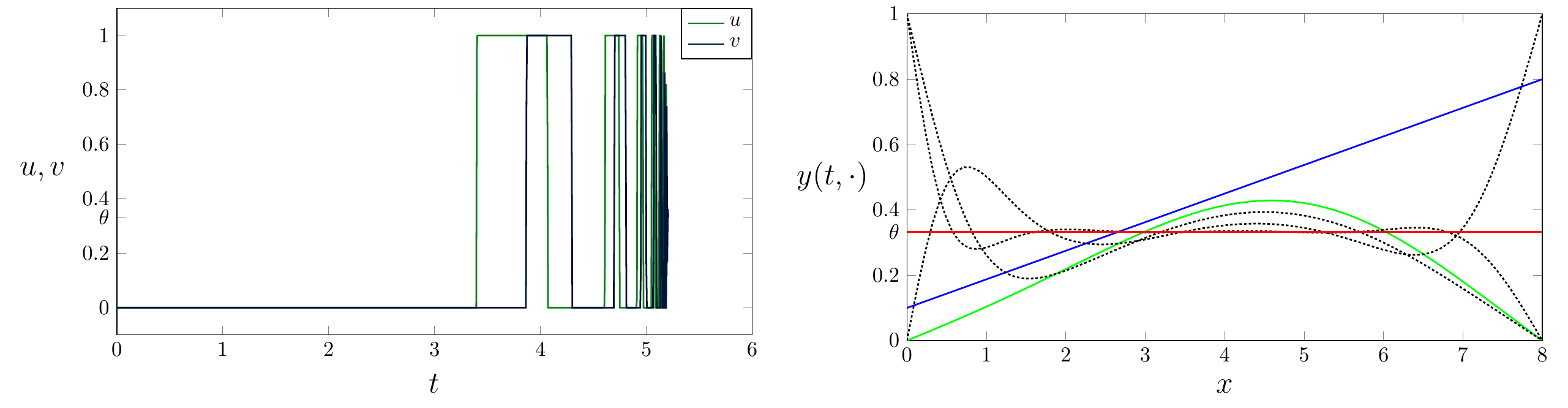}
\end{center}
\caption{Optimal controls $u$, $v$ (lower) and resulting state $y(t, \cdot)$ from time $0$ (in blue), to $t_f$ (in red) with intermediate times $0.6t_f$ (in green), $0.7t_f$, $0.8t_f$, $0.9t_f$ (in dashed line).}
\label{MinimalTime_TwoSides}
\end{figure}

As already stressed, simulations suggest that it is possible to control the system to $\theta$ with the same controls on both sides. A simulation of the minimal time problem with the same control $u$ at $x=0$ and $x=L$ is presented in Figure~\ref{MinimalTime_OneSide}. 
For this non-symmetric initial condition, the minimal time is about $5$ times the minimal time with two controls, showing that the two degrees of freedom strongly accelerate the convergence to~$\theta$. The strategy however remains the same: the control is bang-bang equal to $0$ for a long time, and then chatters around the final time. After the first phase with null control, the state is almost~$0$. The oscillations then gradually fill up the state $\theta$, starting from the middle. 

\begin{figure}[h!]
\begin{center}
\includegraphics[scale=0.58]{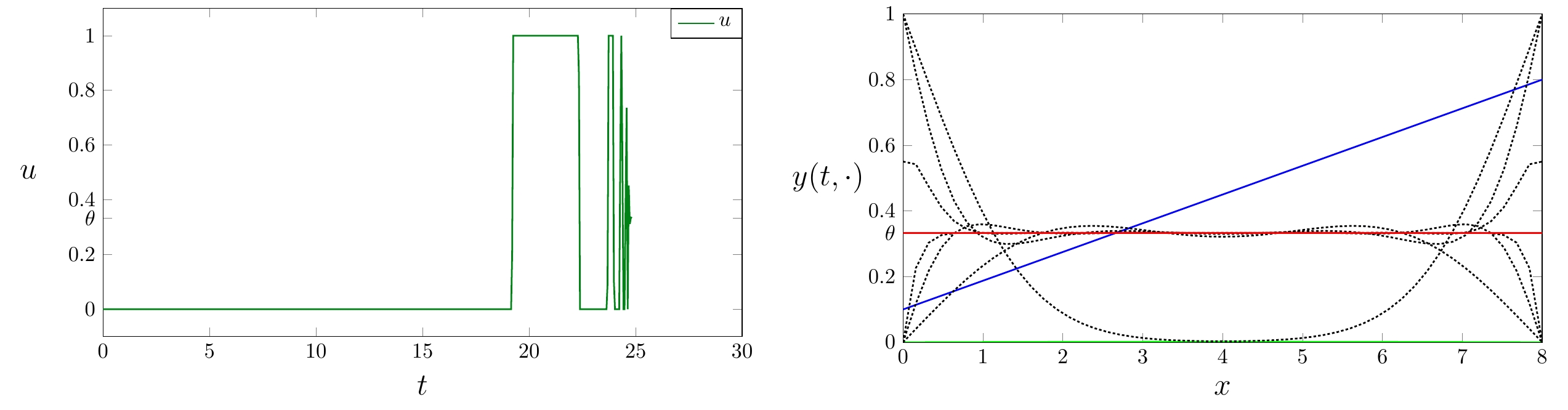}
\end{center}
\caption{Optimal control $u$ and resulting state $y(t, \cdot)$ from time $0$ (in blue), to $t_f$ (in red) with intermediate times $0.7t_f$ (in green), and several others from $0.95t_f$ to $t_f$ (in dashed line).}
\label{MinimalTime_OneSide}
\end{figure}

Finally, let us mention that for symmetric initial conditions, simulations for the minimal time problem (not shown here) exhibit symmetric controls $u(t) = v(t)$.

\subsection{Comments and perspectives}

\textbf{Fully parametrized model.} \;
Explicitly considering the diffusion parameter $\mu$, namely 
\begin{equation}
y_t  - \mu y_{xx} = f(y), 
\end{equation}
it is easily seen after change of variables that the thresholds $L^\star$, $L_\theta$ scale like $\sqrt{\mu}$. 
When exponential convergence holds as given by Lemma \ref{PropEstimate}, the rate of convergence is the first eigenvalue of the Dirichlet Laplacian and thus scales like $\mu$. 

\textbf{Other boundary conditions and steady states.} \;
As a byproduct of our analysis, we also have proved results for 
\begin{equation}
\begin{cases}
\label{ModelNeumann}
y_t  - y_{xx} = f(y), \\
y(t,0) = u(t),  \; y_x(t,L) = 0,\\
y(0) = y_0,
\end{cases}
\end{equation}
namely the system where there is only one control at $x=0$, while a Neumann boundary condition is enforced at the other end of the domain. Indeed, the same phase plane analysis shows that it is controllable towards $0$ in infinite time if and only if (putting $u(t)=0$ at the left end) $L \leq \f{L^\star}{2}$ in the monostable case ($L < \f{L^\star}{2}$ in the bistable case, respectively).

Simulations not shown here suggest that this system can be controlled to $\theta$ if $L<\f{L^\star}{2}$ in the bistable case, and it is an open problem to prove it (as our control strategy requires to act on both ends). 

Also note that our overall strategy would also work it we had Neumann controls instead of Dirichlet controls. It can also be used to reach other stationary states, while the strategy as well as possible obstacles and corresponding threshold values are all readable on the phase plane. 

\textbf{The multi-dimensional case.} \;
Understanding what happens in the previous case would be critical in view of tackling the problem in higher dimension. It is indeed natural to think of situations where the control acts only on a part of the boundary, while the rest of the boundary is endowed with Neumann conditions. For such non-homogeneous boundary conditions, stability results would be required to carry out an analysis in the spirit of ours.

If the control acts on the whole boundary, the problem of controllability towards $0$ again leads to analyzing whether only the trivial solution solves the stationary problem, because the result of Matano has been generalized~\cite{Simon1983}. Then, the threshold phenomenon is already known~\cite{Lions1982}. In this work, it is stated for 
\begin{equation*}
\begin{cases}
- \Delta y  =   \lambda f(y) \text{ in } \Omega, \\
 y = 0 \text{ on } \p \Omega
\end{cases}
\end{equation*}
where the parameter related to the domain size is $\lambda$. However, there are up to our knowledge no explicit formulae for the threshold value, although bounding like in Subsection \ref{SubsectionEst} still works.

For the control towards $1$ and in the monostable case (H1), the Lyapunov functional introduced in Remark \ref{Lyapunov} works in arbitrary dimension~\cite{Pouchol2018}. For the control towards $\theta$ in the bistable case (H2) with the static strategy of putting $\theta$ on the whole of $\p \Omega$, there is also a threshold as can be proved thanks to the above result.

\textbf{Non-local extension.} \;
A possible extension of the monostable case is to replace the classical nonlinearity $f(y) =y(1-y)$, by a non-local one, namely \[f(x,y) = y\left(1-\int_0^L K(x,z) y(z) \,dz\right)\]
where $K$ is a kernel accounting for interactions between individuals at positions $x$ and $z$ and the equation is usually called the non-local Fisher-KPP equation~\cite{Pouchol2018}. Since the corresponding stationary equation depends on $x$, the phase portrait technique does not apply and extending the controllability properties considered in this paper is a completely open problem. 

\textbf{Controllability above the thresholds.} \;
It is natural to analyze which initial conditions can still be brought to $0$ or $\theta$, when $L \geq L^\star$. Let us give a simple negative result in this direction. To answer the question, the multiplicity of stationary solutions is critical and there are many results~\cite{Lions1982}. Generally, these solutions are ordered, and thus, the maximal one will attract all initial conditions that are uniformly above it, by the comparison principle and Matano's theorem. Going more deeply would require to analyze the basin of attraction of each stationary solution.

There are also some initial data for which we can prove controllability towards  $\theta$ with our staircase strategy. Indeed, lack of controllability for $L>L^\star$ comes from the impossibility to make the first step of the proof work, namely to let any initial condition reach the state $y_{init}$ close to $0$. The second part, however, still works: $\theta$ is linked to any steady state in the region $\Gamma$, independently of $L$. Thus, for any $L>0$, we can control any of these steady states to $\theta$, as well as any initial condition close to one of them (it has to be close enough for the local controllability argument to apply).

\textbf{Open-loop or feedback control.} \; 
Our control strategy towards $\theta$ is completely constructive and open-loop since we first put low controls $\varepsilon$ on both sides up until the trajectory is close enough to $y_{init}$, and the waiting time can be taken to be independent of the initial condition, as proved in Proposition \ref{UniformTime}. The staircase phase requires controls achieving local controllability, which are also open-loop and constructive: these controls are indeed obtained by the HUM method, namely the minimization of an appropriate functional~\cite{Lions1988}. 

The strategy can also be defined in feedback form, since the first phase can be made to last up until the trajectory is close enough to $y_{init}$, while the staircase one could also be designed in feedback form, by adapting the results of~\cite{Coron2004}.

\textbf{Minimal time.} \;
For a given initial condition, theoretically estimating the minimal time for controllability towards $\theta$ and corresponding controls is an open problem in our semilinear setting, since spectral estimates specific to the linear case (used in~\cite{Loheac2017}) are not available. 

\textbf{Same controls on both sides.} \;
The controllability properties proved in the present paper require acting with different controls on both sides only for the state $\theta$, although numerical simulations suggest that these properties also hold with $u=v$. Proving it is an open problem, and an interesting question is whether this could also be done thanks to a path of steady states. 

\textbf{Path of steady states.} \; 
More generally, the staircase strategy is instrumental in our proof and the underlying path of steady states is obtained by phase plane analysis. Thus, in view of tackling problems in dimension higher than $1$, finding alternatives to the phase plane approach is a relevant open problem. 

A possible approach to develop intuition on possible paths of steady states is to build an optimal control problem forcing the system to remain close to stationary states, for example by adding a constraint like
\[| \Delta y + f(y)| \leq \e, \] where $\e$ is small. However, this constraint alone would be too restrictive since there must be a first phase which consists in reaching some steady state on the path. These two requirements make the construction of such optimal control problems highly non-trivial.

\paragraph{Acknowledgment.}
The authors would like to thank Martin Strugarek and Grégoire Nadin for fruitful discussions on bistable equations and maximum principles. 
\newline
This research was supported by the Advanced Grant DyCon (Dynamical Control) of the European Research Council Executive Agency (ERC), the MTM2014-52347 and MTM2017-92996 Grants of the MINECO (Spain), the ICON project of the French ANR-16-ACHN-0014 and the AFOSR Grant FA9550-18-1-0242 "Nonlocal PDEs: Analysis, Control and Beyond".
{
\bibliography{ControlPhasePortrait.bib}
\bibliographystyle{acm}}

\end{document}